\documentclass[a4paper]{article}
\usepackage{amsmath,amsfonts,amssymb}
\usepackage{graphicx, epsfig}
\newtheorem{theorem}{Theorem}[section]
\newtheorem{lemma}[theorem]{Lemma}

\newtheorem{prop}[theorem]{Proposition}

\newtheorem{claim}[theorem]{Claim}

\def\Kerek{Ker\'ekj\'art\'o} 

\newenvironment{proof}[1][Proof]{\textbf{#1.} }
{\hfill\rule{0.5em}{0.5em}\medskip}
\newenvironment{proof*}[1][Proof]{\textbf{#1.} }{}

\def\epsilon{\varepsilon}
\def\WB{$\omega$-bounded }

\begin{document}

\title{Euler-Poincar\'e obstruction for pretzels\\
with long tentacles \`a la Cantor-Nyikos }

\author{Alexandre Gabard}
\maketitle

\newbox\quotation
\setbox\quotation\vtop{\hsize
7.8cm \noindent

\footnotesize {\it Poincar\'e hat zuerst die Frage nach dem
Gesamt\-verlauf der reellen L\"osungen von
Differentialgleich\-ungen mit topo\-lo\-gischen Mitteln
behandelt.}

\noindent Hellmuth Kneser, 1921, in {\it Kurvenscharen auf
geschlossenen Fl\"achen {\rm \cite{Kneser_1921}}}.




}

\hfill{\hbox{\copy\quotation}}

\medskip

\newbox\abstract
\setbox\abstract\vtop{\hsize 12.2cm \noindent


\footnotesize \noindent\textsc{Abstract.} We present an avatar
of the Euler obstruction to foliated structures on certain
non-metric surfaces. This adumbrates (at least for the
simplest 2D-configurations) that the standard mechanism---to
the effect that the devil of algebra sometimes barricades the
existence of angelic geometric structures (obstruction theory
more-or-less)---propagates slightly beyond the usual metrical
proviso. Alas, the game is much more conservative than
revolutionary: in particular we enjoyed retrospecting at
Poincar\'e's argument of 1885 (announced in 1881).}

\centerline{\hbox{\copy\abstract}}

{\small \tableofcontents}

\section{Introduction}\label{sec1}

\subsection{Statement and nomenclature}

In a previous paper \cite{Gabard_2011_Ebullition}, we explored
foliated surfaces at the large scale by
contemplating (rather passively) how several
classic paradigms (like those of Poincar\'e-Bendixson or
Haefliger-Reeb)
transpose non-metrically.
A {\it thermodynamical metaphor\/} was found
convenient
to synthesize  a
body of
disparate results (including the irrational toric windings of
Kronecker, the allied labyrinths of Dubois-Violette, Franks,
Rosenberg and the surgeries of Peixoto--Blohin).
The metaphor
implicates the familiar solid-liquid-gaseous phases as
follows:

$\bullet$ If
one warms
sufficiently the fundamental group (by increasing its
rank~$r$,
say via
iterated punctures in a closed surface), then as the
temperature is high enough ($r\ge 4$) {\it metric\/} surfaces
are always {\it transitively\/} foliated
(by a dense leaf). (Observationally, such random motions
are  expected to occur
when
tracing out a curve following
some primate's fingerprints leading quickly to some
complicated `labyrinth' filling almost all of
the hand epiderm.)

$\bullet$ Conversely at low temperatures ($r\le 1$), e.g., in
the simply-connected
case,
the situation is completely frozen: {\it
any foliated structure is intransitive},
{\it regardless\/} of the metric proviso.

$\bullet$ Between the
first gaseous-volatile regime ($r\ge 4$) and the second
solid-frozen state ($0\le r\le 1$), we observe for $2\le r\le
3$ an intermediate liquid-phase, where the transitivity issue
truly depends on the detailed
topology.

In particular, each closed surface has a {\it critical
temperature\/}
at which it starts an ebullition, namely the least integer $k$
such that the $k$-times punctured surface permits a transitive
foliation (i.e., with a dense leaf). (This
boiling temperature is
computed in
\cite{Gabard_2011_Ebullition} for all surfaces except
the Klein bottle. We presume the answer
to be more-or-less implicit in the works of Hellmuth Kneser
from the 1920's (\cite{Kneser_1921}, \cite{Kneser24}), but as
yet we have not
assembled the details.)

Besides, still in \cite{Gabard_2011_Ebullition}, the question
of an avatar of the Euler-Poincar\'e obstruction was left
dormant. The present  note aims to
remedy this gap by showing:

\begin{prop}\label{Euler-obstruction:prop} An $\omega$-bounded
surface
of negative Euler characteristic
\hbox{$(\chi<0)$} cannot be foliated.
\end{prop}

$\bullet$ Here, a {\it surface\/}
is as usual a 2-dimensional
manifold, that is a Hausdorff space everywhere locally
homeomorphic to the number-plane ${\Bbb R}^2$, yet without
imposing (a priori) a global metric
compatible with the topological structure.

$\bullet$ A space is
{\it $\omega$-bounded} if any countable subset has a compact
closure. This concept
turned out to be a
vivid substitute to compactness in the non-metric realm,
especially in view of {\it Nyikos' bagpipe theorem\/}
\cite{Nyikos84}. The latter extrapolates widely the {\it
classification of compact surfaces} initiated by M\"obius
circa 1860 published 1863 \cite{Moebius_1863} (after loosing
an international Parisian
contest for linguistical
issues) and subsequently revisited
(and generalized) by a long list of workers including Jordan,
Klein, Dyck 1888 \cite{Dyck_1888}, Dehn-Heegaard 1907, Brahana
1921.
If some compact manifold should
geometrize the myth\footnote{Courtesy of Claude Weber for
pointing out a text of Borges, establishing a one-to-one
correspondence between
conceptions of times through the ages and
1D-manifolds; e.g. $S^1\leftrightarrow \text{``l'\'eternel
retour''}$.} of a {\it finite universe},
\hbox{$\omega$-boundedness}
posits
the scenario of
a possibly infinite, yet {\it inextensible universe}. It is
indeed easy to
convince that $\omega$-boundedness implies
sequential-compactness, implying in turn {\it maxi\-fold},
i.e., a manifold which cannot be
continued to a strictly larger (connected) manifold of the
same dimensionality. This
little remark may help to grasp
the
substance (at least the aesthetical value) of
$\omega$-boundedness. The simplest (non-metric) prototype of
an $\omega$-bounded manifold is the {\it long line\/} ${\Bbb
L}$ or, in the bordered case, the {\it closed long ray\/}
${\Bbb L}_{\ge 0}$ (discovered  by Cantor in 1883).
Those spaces are both grandiose and claustrophobic:
one cannot find
a {\it Fluchtweg\/} to infinity in
denumerable time.

$\bullet$ {\it Foliations} refer as usual to those geometric
structures microscopically modelled after the partition of the
plane into
parallel
lines (we
focus on the surface case for simplicity).  {\it Regular
family of curves} is an older synonym, employed by \Kerek{
}1925~\cite[p.\,111]{Kerekjarto_1925} and Whitney
1933~\cite{Whitney33}.
Gauss in 1839
uses the term {\it Liniensystem} (in {\it Allgemeine Theorie
des Erdmagnetismus} \cite[p.\,135]{Gauss_1839}) and the German
speaking literature (especially \Kerek{ }1923
\cite[p.\,249]{Kerekjarto_1923}, Kneser 1921
\cite{Kneser_1921}, 1924 \cite{Kneser24}) used the jargon {\it
Kurvenschar}. Later Ehresmann and Reeb coined {\it
feuilletage}, from which the
modern nomenclature was derived.
For our
concern, the
real issue is that the foliated concept
(and likewise the manifold idea), being purely local in
nature, are not
imprisoned
in the metric realm.

$\bullet$ Finally, in the $\omega$-bounded
context the (singular) homology is a priori known to be of
finite type (cf.
optionally the discussion in \cite{Gabard_2011_Hairiness}). In
particular the
{\it Euler characteristic} ($\chi$=alternating sum of Betti
numbers starting with $b_0$
signed positively!\footnote{In the older literature it seems
that the sign of $\chi$ was the opposite one! Courtesy of
M.~Kervaire's
lecture notes borrowed from L.
Bartholdi.}) is {\it finite\/}
(unambiguously defined despite the
severe---but not so dramatic---absence of triangulation).
Actually a simple argument (reproduced below) shows
the
characteristic of the surface
to coincide with that of {\it Nyikos' bag} (which is a compact
bordered surface).

Admittedly, the theory of nonmetric manifolds is far from
popular,
yet
its bad reputation seems to
be slightly overdone,
since ironically much of the game is just a matter of
transposing metrical truths. The present note is no exception.
In practice, there is often an automatic
transfusion (of truths) from Lindel\"of subregions to the
whole manifold (e.g., for Jordan separation or the Schoenflies
bounding disc property, orientability, etc., cf. optionally
\cite{GaGa2010} and \cite{Gabard_2011_Ebullition}), whereas in
the present case we rather use a
{\it d\'evissage} into metric sub-pieces (cf.
Fig.\,\ref{Bag:fig}) where the
obstruction is classical.

\subsection{Metric background and historiographical links}

Specifically, our proof of (\ref{Euler-obstruction:prop})
relies on
the following (metrical) background,
reviewed subsequently so as not to relegate our
issue behind a mountain of preliminaries:

(1) {\it Poincar\'e's index formula\/} or
the variant thereof for foliations. (Main contributors:
Poincar\'e 1880 \cite{Poincare_1880}, 1881
\cite{Poincare_1881}, 1885 \cite[1885,
p.\,203--8]{Poincare_1885}, Dyck 1885
\cite[p.\,317--320]{Dyck_1885}, 1888 \cite[p.\,462--3;
499--501]{Dyck_1888}, Brouwer and Hadamard 1910
\cite{Hadamard_1910},
Hamburger 1924 \cite[p.\,58--62]{Hamburger_1924}, Hopf 1926
\cite{Hopf_1926} and Lefschetz about
the same period.
It is perhaps fair to recall that this circle of ideas had
been partly anticipated
by Gauss
and  Kronecker 1869, e.g., via the {\it Curvatura integra},
just to name two among several other
forerunners
carefully listed in Dyck
1888 \cite[p.\,463]{Dyck_1888} (e.g., Gauss 1839, Reech 1858,
M\"obius 1863, Poincar\'e 1881--86, Klein 1882, Betti 1885,
etc.). See also the recent historiography in
Mawhin 2000 \cite{Mawhin_2000},
reminding in particular Hermite's r\^ole, as an interface from
Kronecker to Poincar\'e.)

(2) {\it Bendixson-Hamburger's index formula} computing the
Poincar\'e index in terms of the local phase-portrait
(already (???) in Enrico\footnote{Little joke to
imitate Enriques--Chisini.} Poincar\'e 1885
\cite[p.\,203]{Poincare_1885}, Bendixson 1901
\cite[p.\,39]{Bendixson_1901}, Hamburger 1922
\cite{Hamburger_1922} (\cite{Hamburger_1924},
\cite{Hamburger_1940}), \Kerek{ }1925 \cite{Kerekjarto_1925},
etc.)
To elucidate the above question marks, we may agree perfectly
with the following comment
of Mawhin
\cite[p.\,118]{Mawhin_2000}:

\smallskip
{\small

``It is of interest to notice that Poincar\'e introduces
here a further definition for the {\it index of a cycle},
without worrying about proving its equivalence or relation
with the previous definitions. This time, the index is defined
as $\frac{E-I-2}{2}$, where $E$ (resp. $I$) is the number of
exterior (resp. interior) tangency points of the vector field
to the cycle. Proving the equivalence between this new
definition and the previous one is essentially the contents of
what is called today the {\it Poincar\'e-Bendixson index
theorem} [Bendixson, 1901].''

}

\smallskip

(3) Some geometric topology \`a la Schoenflies circa 1906,
plus some
variants due to R. Baer 1928, R.\,J. Cannon 1969. Compare
Gabard-Gauld 2010 \cite{GaGa2010} for a
pedestrian exposition of the fact that {\it any null-homotopic
Jordan curve in a
surface bounds a disc}. This holds {\it universally} (without
metric proviso).

\subsection{Related
phenomenology
and an application
\`a la Kaplan-Haefliger-Reeb}

For comparison,
{\it non-singular flows}
(fixed-point free ${\Bbb R}$-actions) are also regulated by an
Euler obstruction namely $\chi\neq 0$ (Gabard
2011~\cite{Gabard_2011_Hairiness}). This
weaker
numerical condition is {\it not} suited to
foliations,  as the {\it long plane\/} ${\Bbb L}^2$ (Cartesian
square of the long line ${\Bbb L}$),
despite having non-zero
$\chi=b_0-b_1+b_2=1-0+0=1$, {\it still\/} foliates (e.g., in
the trivial fashion by parallel long lines).

\begin{figure}[h]
\centering
    \epsfig{figure=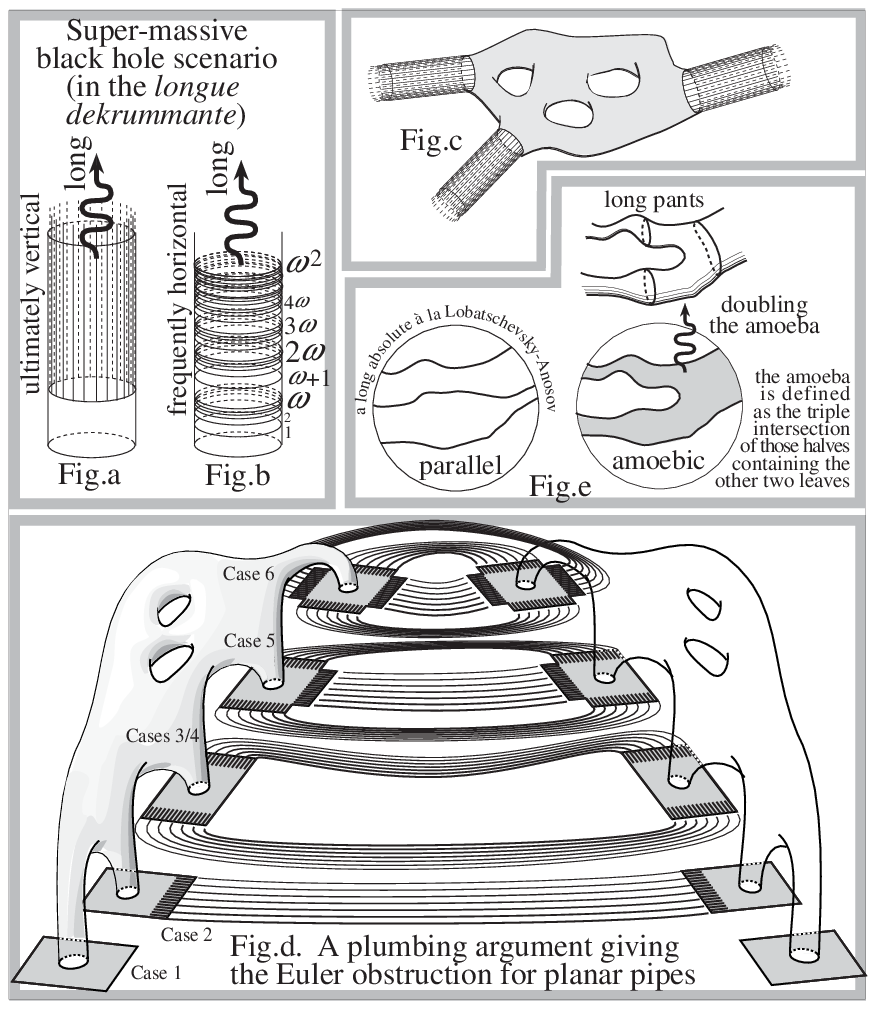,width=122mm}
\vskip-5pt\penalty0
  \caption{\label{Massive:fig}
  Super-massive black hole and plumbing with spaghetti}
\vskip-5pt\penalty0
\end{figure}

The converse of (\ref{Euler-obstruction:prop}) fails: a
surface with $\chi\ge 0$ does not necessarily foliate. The
simplest example is probably the sphere $S^2$ (as well-known
since Poincar\'e 1880 \cite{Poincare_1880}). For a non-metric
prototype,  we may consider a {\it long cigar\/} $S^1\times
{\Bbb L}_{\ge 0}$ (=circle crossed by the closed long ray)
capped off by a 2-disc. The resulting surface (resembling a
long glass) lacks foliations (cf. Baillif {\it et al.\/} 2009
\cite{BGG1}). This derives from the {\it super-massive black
hole scenario\/} materialized
by Cantor's long ray, according to which a (finally violent)
aspiration of leaves in the semi-long cylinder $S^1\times
{\Bbb L}_{\ge 0}$ occurs either as an ultimately
vertical collection of straight long rays or as a compulsive
infinite repetition of horizontal circle leaves parametrized
by a closed unbounded set (compare
Fig.\,\ref{Massive:fig}a,b). Thus for many
toy examples of $\omega$-bounded surfaces, say those
constructed by inserting long cylindrical pipes
into a pretzel (Fig.\,\ref{Massive:fig}c),
Proposition~\ref{Euler-obstruction:prop}
boils down
to the classic compact Euler obstruction.
With some little more work (plumbing) the same
trick applies to planar long pipes modelled after the long
plane (compare Fig.\,\ref{Massive:fig}d showing how to
``plumb'' with a replica the
bagpipe surface suitably truncated according to the
6 possible asymptotic
patterns described in
\cite{BGG1}; arithmetical details left to the indulgent
readers).
In short, the
little innovation of the present result
(\ref{Euler-obstruction:prop}) is that,
while presupposing
 no explicit knowledge
of the {\it pipes\/}
(whose
biodiversity
overwhelms any classification
scheme),
it still affords a qualitative prediction
in close accordance to
our metric
intuition.

An
application of (\ref{Euler-obstruction:prop}) can be given to
foliated structures
on simply-connected
$\omega$-bounded surfaces.
By a (non-metric) extension of Kaplan/Haefliger-Reeb's theory
(cf. \cite{Gabard_2011_Ebullition}),
{\it any leaf in a simply-connected surface separates the
surface} (\`a la Jordan). Thus given a configuration of 3
leaves they can either (compare Fig.\,\ref{Massive:fig}e) be
{\it parallel} (with one central leaf separating the other
two) or
bound an {\it amoeba} (if no leaf disconnects the other two).
In the $\omega$-bounded case, the
amoebic option cannot occur since the doubled amoeba yields a
{\it long pants\/} (Fig.\,\ref{Massive:fig}e) with $\chi=-1$,
hence not foliable by (\ref{Euler-obstruction:prop}). It
follows that the leaf-space is necessarily a {\it Hausdorff\/}
\hbox{$1$-manifold} (since given 2 leaves, any leaf chosen in
between imposes a separation \`a la Jordan, implying
a
separation \`a la Hausdorff in the quotient leaf-space).

Perhaps
it is
reasonable to expect higher-dimensional extensions of
(\ref{Euler-obstruction:prop}) for foliations of dimension- or
codimension-one. Maybe, one should first concretize
Nyikos-Gauld's grand programme of a 3D-bagpipe philosophy
(probably by now ``harpoonable'' via the Poincar\'e conjecture
of Perelman).

\section{Proof of the proposition}

 The proof of (\ref{Euler-obstruction:prop}) can
 be given two slightly different
flavors
by arguing either with {\it foliations} or (the
allied) {\it flows} (continuous ${\Bbb R}$-actions).

First, we may reduce to the case of
an oriented foliation by
passing to the double cover orienting the foliation (cf. e.g.,
\cite[4.1]{Gabard_2011_Ebullition}). This
doubles the Euler characteristic $\chi$, thereby preserving
its negativity.
The
(fundamental) theorem of \Kerek-Whitney (1925
\cite{Kerekjarto_1925}, 1933 \cite{Whitney33})---to the effect
that an oriented foliation admits a compatible flow---fails
non-metrically (\cite{GabGa_2011}), but
applies to
Lindel\"of (equivalently metric)
subregions. Hence the theory of flows can still be
advantageously exploited after
some precautions. In its foliated variant, the proof below is
pure but
uses the index formula for line-fields
(involving {\it semi-integral} indices), whereas working with
flows requires (beside \Kerek-Whitney) some ad hoc (but
classical) mechanisms for slowing down flow lines (Beck's
technique 1958 \cite{Beck_1958}), as we shall
discuss at the appropriate moment.

\medskip
\begin{proof}[Proof of \ref{Euler-obstruction:prop}]
Here is the common core of the argument (quite regardless of
which viewpoint is adopted, and concretely which version of
the index formula is applied). The first ingredient is Nyikos'
theorem \cite{Nyikos84} according to which
an \WB surface
has a {\it bagpipe decomposition}
$$
M=B\cup \textstyle\bigcup_{i=1}^n P_i.
$$
This is to mean that there is a compact bordered surface
$B\subset M$ (referred to as the {\it bag}) such that the
components of $M-{\rm int} B$ are bordered surfaces $P_i$ (the
{\it pipes}) which filled by a disc become simply-connected
surfaces $\Pi_i$. Hence $\chi(\Pi_i)=1-0+0=1$ (recall the
vanishing of the top-dimensional homology of an open Hausdorff
manifold, cf. e.g., Samelson 1965
\cite{Samelson_1965-homology}). It follows
by additivity of the characteristic (formally Mayer-Vietoris)
that $\chi(P_i)=0$ and
in turn that
\begin{equation}\label{char-Bag:Eq}
\chi(M)=\chi(B).
\end{equation}

As $B$ is compact its boundary $\partial B$ consists of
finitely many circles $C_i$ (say $n$), which we call the {\it
contours} of $B$. Each contour $C_i$
has a tubular neighborhood $U_i\approx S^1\times[-1,+1]$ (also
interpretable as a {\it bicollar}) whose border $\partial U_i$
splits in 2 circles $C_i^+$ and $C_i^-$. We agree that the
plus version $C_i^+$ is the one lying in the pipe $P_i$,
whereas the minus $C_i^-$ are all
in the bag $B$
(compare Fig.\,\ref{Bag:fig}, top part).

\begin{figure}[h]
\centering \epsfig{figure=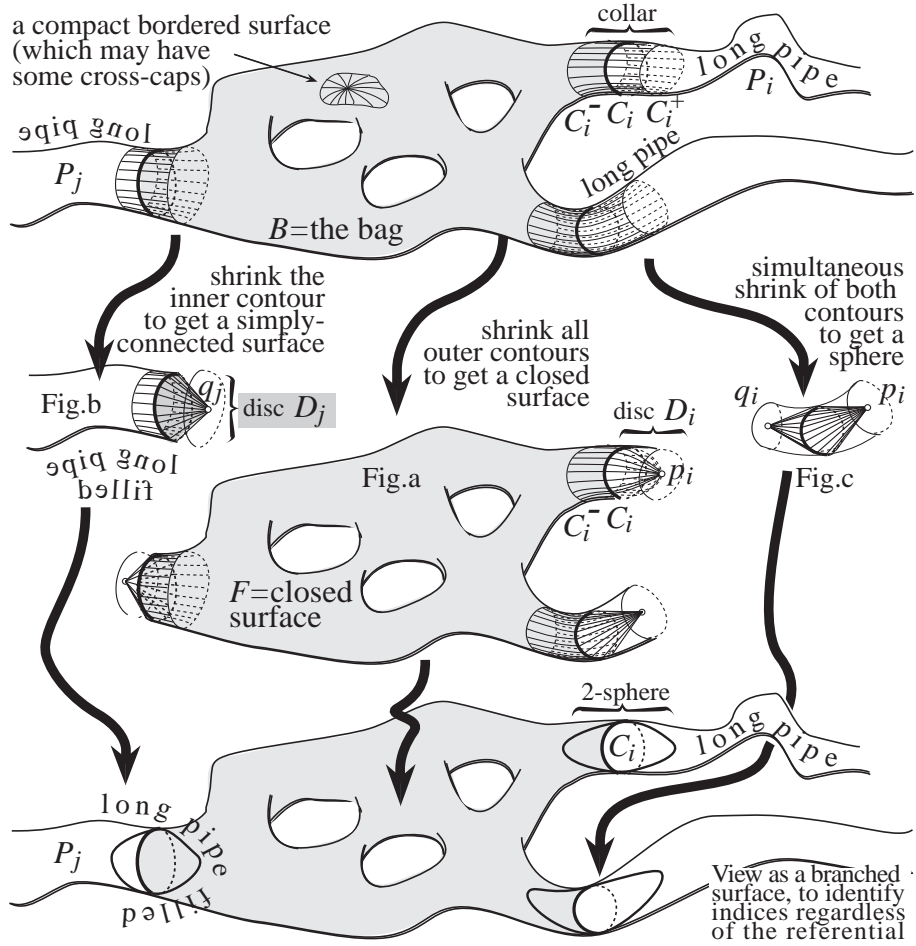,width=122mm}
\caption{\label{Bag:fig} Dissecting a bagpipe into elementary
pieces}
\end{figure}

Collapse this bagpipe configuration $M=B\cup \bigcup_{i=1}^n
P_i$ in essentially 3 ways:

(a)
shrink the outer contours $C_i^+$ to points $p_i$, to produce
a closed surface $F$ homeomorphic to the bag $B$ capped off by
$n$ discs (Fig.\,\ref{Bag:fig}a), hence
\begin{equation}\label{Eq:F-filled}
\chi(F)=\chi(B)+n;
\end{equation}

(b)
shrink the inner contours $C_i^-$ to points $q_i$, to get $n$
surfaces $\Pi_i$ respectively homeomorphic to the pipe $P_i$
capped off by a disc $D_i$ (Fig.\,\ref{Bag:fig}b), hence
simply-connected;

(c)
shrink both contours $C_i^+$, $C_i^-$ simultaneously  to
points $p_i$, resp. $q_i$, to get $n$ spaces homeomorphic to
the sphere $S^2$ (Fig.\,\ref{Bag:fig}c).

In each cases it is understood that
the foliated structure undergoes the same shrinkages,
thereby creating isolated singularities precisely at those
points where
circles are collapsed. If
one prefers to argue with flows, first choose a compatible
flow on some open (metric) neighborhood of the bag $B$, and
slow it down
to
be at rest on the circles $C_i$. This
is achieved via {\it Beck's technique} (1958
\cite{Beck_1958}).
Hence the points (=collapsed circles) are the unique
rest points of the flow, which can therefore be
dissociated  into several flows over the elementary pieces of
the dissection given by Fig.\,\ref{Bag:fig}.

Applying the Poincar\'e index formula (cf. e.g.,
(\ref{index-formula}) below) in those varied
surfaces, we get from
collapse (a)
\begin{equation}\label{Eq:Poincare}
\textstyle\sum_{i=1}^n i(p_i)=\chi(F),
\end{equation}
whereas
collapse (c) gives
\begin{equation}\label{Eq:sphere}
i(p_i)+i(q_i)=\chi(S^2)=2.
\end{equation}

\begin{claim}\label{claim}
In the filled pipes $\Pi_i=P_i\cup D_i$ generated by operation
{\rm (b)},
the following estimate holds:
\begin{equation}\label{Eq:estimates}
i(q_i)\le 1.
\end{equation}
\end{claim}

We postpone
the
verification of (\ref{claim})
to complete first
the proof of
(\ref{Euler-obstruction:prop}). Assembling those five
equations \eqref{char-Bag:Eq}--\eqref{Eq:estimates} we get
$$
\underbrace{\sum_{i=1}^n\underbrace{(2-i(q_i))}_{\ge 1 \text{
by } \eqref{Eq:estimates}}}_{\ge
n}\buildrel{\eqref{Eq:sphere}}\over{=}\sum_{i=1}^n
i(p_i)\buildrel{\eqref{Eq:Poincare}}\over{=}\chi(F)\buildrel{\eqref{Eq:F-filled}}\over{=}\chi(B)+n
\buildrel{\eqref{char-Bag:Eq}}\over{=}\chi(M)+n,
$$
violating the assumption $\chi(M)<0$.
\end{proof}

\begin{proof}[Proof of Claim~\ref{claim}] As the filled pipe $\Pi_i:=P_i\cup D_i$ is
simply-connected,
let us imagine it pictured in the plane (yet a spicy version
thereof going at infinity in a funny way).
This depiction has no intrinsic meaning,
except reminding us that the Schoenflies theorem holds true in
{\it every\/} simply-connected surface. (Recall
from \cite{GaGa2010} that any Jordan curve in a
simply-connected surface bounds a 2-disc, regardless from any
metric assumption.)

The key trick (quite omnipresent in  Bendixson 1901
\cite{Bendixson_1901}, or Mather 1982 \cite{Mather_1982}) is
to
pay special attention
at leaves both of whose ends converge to the `origin' $q_i$.
Call such leaves {\it loops}, for short.

\begin{figure}[h]
\centering
    \epsfig{figure=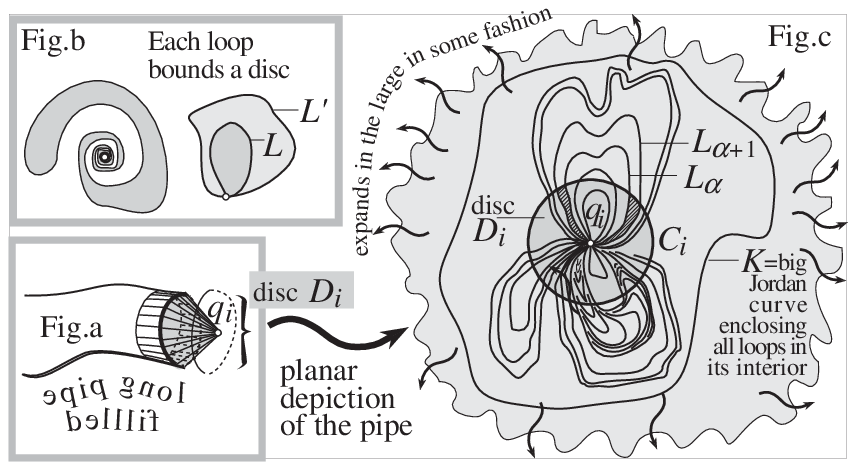,width=122mm}
  \caption{\label{Loops:fig}
  Blocking the proliferation of loops in a big circle $K$}
\end{figure}

Any loop can be completed to a Jordan curve by adding the
point $q_i$.
Hence, according to the (universal) Schoenflies theorem
\cite{GaGa2010}, any loop $L$ bounds a unique disc $D_L$ in
$\Pi_i$ (Fig.\,\ref{Loops:fig}b). Thus there is a
partial order on the set $\cal L$ of all loops by
decreeing $L\le L'$ whenever the inclusion $D_L\subset D_{L'}$
holds for the corresponding bounding discs. (Note: $D_L$ is
essentially what Bendixson calls a {\it nodal region}.)

Now observe that a well-ordered chain of loops $(L_{\alpha})$
has at most
countable `height', i.e. cardinality. Otherwise looking at the
successive symmetric differences ${\rm int }
D_{L_{\alpha+1}}-D_{L_{\alpha}}$ inside the disc $D_i$ gives
uncountably many pairwise disjoint open sets in the disc,
against its separability (compare Fig.\,\ref{Loops:fig}c).

Likewise there is at most countably many loops
pairwise incomparable w.r.t. the order $\le$ on $\cal L$. Thus
we can find a countable sequence of loops $(L_n)_{n<\omega}$
{\it cofinal} in $(\cal L, \le)$, i.e. for each $L\in {\cal
L}$ there is an integer $n<\omega$ such that $L\le L_n$. Then
the union of all loops $\Lambda:=\bigcup_{L\in\cal L} L$ and
the union $\Delta:=\bigcup_{n<\omega} D_{L_n}$ have the same
closures. Since the latter set,
$\Delta$, is $\sigma$-compact (hence Lindel\"of), it has a
compact closure by $\omega$-boundedness.

So the closure $\overline{\Lambda}$ is a compactum in a
simply-connected surface (which is not the sphere),
hence contained in a large disc $D$ (compare Gabard 2011
\cite[Lemma~2.34]{Gabard_2011_Ebullition}) whose boundary
contour $\partial D=K$ is a large circle enclosing
$\overline{\Lambda}$ in its interior. By construction the
circle $K$ does not
encounter any loop. Further we may assume that $K$ encloses
also the disc $D_i$ in its interior, and so (by Schoenflies
again)
$K$ is freely homotopic to $C_i$ (in fact the difference
$D-{\rm int} D_i$ is an annulus).
Accordingly,
one may compute the index $i(q_i)=i(q_i,C_i)$ w.r.t. to the
curve $K$. Finally, classical index theory (cf.
Lemma~\ref{index-small-if-no-loops} below) shows that $i(q_i,
K)\le 1$.
This is the desired estimate.
\end{proof}

\section{Memento of 2D-index theory}

\subsection{Index of an isolated singularity
(Poincar\'e, Bendixson, Hamburger,
etc.)}

Without developing the full theory in a coherent fashion, we
just recall enough
background to
establish the
next lemma required to complete our argument. From now on, we
switch (as it is quite customary) the generic notation $i$ for
the (Poincar\'e) index to $j\in\frac{1}{2}{\Bbb Z}$ to
emphasize its semi-integral nature.

\begin{lemma}\label{index-small-if-no-loops}
If a foliation (or a flow) on the punctured plane ${\Bbb
R}^2_{*}$ admits a circle $K$ enclosing the origin through
which no leaf is a loop (i.e. a leaf both of whose ends
converge to the puncture). Then the index $j(0, K)$ of the
origin w.r.t. the curve $K$ is $\le 1$.
\end{lemma}

We shall derive this from
the classical formula for the Poincar\'e index due to
Bendixson \cite[p.\,39]{Bendixson_1901},
which
\Kerek{ }{\rm\cite[p.\,108--9]{Kerekjarto_1925}}
assigns to Hamburger 1922 \cite{Hamburger_1922}.
%
%
(Having no access to Hamburger's paper, we tried to
reconstruct a proof,
although
the original is surely more
readable than what to be found bellow.)

\smallskip
{\small {\sc Added in proof.} Meanwhile several presentations
in book forms are available, e.g. Lefschetz 1957
\cite[p.\,222]{Lefschetz_1957}, Hartman 1964 \cite[p.\,173,
Thm 9.2]{Hartman_1964} and Andronov {\it et al.} 1973
\cite[p.\,511]{Andronov_1973}.
{\it Little warning (for `googlers'):\/} what is called below
{\it Hamburger's formula}
(after \Kerek{ }\cite[p.\,108]{Kerekjarto_1925}) is more
commonly known as {\it Bendixson's index formula}, and less
frequently also as the {\it Poincar\'e-Bendixson index
formula}, which is
quite fair in view of the
formula $J=\frac{E-I-2}{2}$
to be found in Poincar\'e 1885 \cite[1885,
p.\,203]{Poincare_1885}.

}

\begin{lemma}\label{Hamburger:lemma} Given a
foliation of the plane with an isolated singularity $p$,
there is always a polygonal circuit $P$ enclosing the
singularity composed of a finite number of leaf-arcs and of
cross-arcs. Call a vertex of this polygon convex or concave
according as the leaf through it can instantaneously move
outside
the polygon or stays in its interior (cf. \,{\rm
Fig.\,\ref{Hamburg:fig}a}, where the shaded region depicts the
residual component of $P$ containing $p$). Let $c$, $c'$ be
the numbers of convex resp. concave vertices. Then the
Poincar\'e index of the singularity $p$ is given by
\begin{equation}\label{Hamburger:formula}
j(p)=1-\frac{c-c'}{4}.
\end{equation}
\end{lemma}

\begin{figure}[h]
\centering
    \epsfig{figure=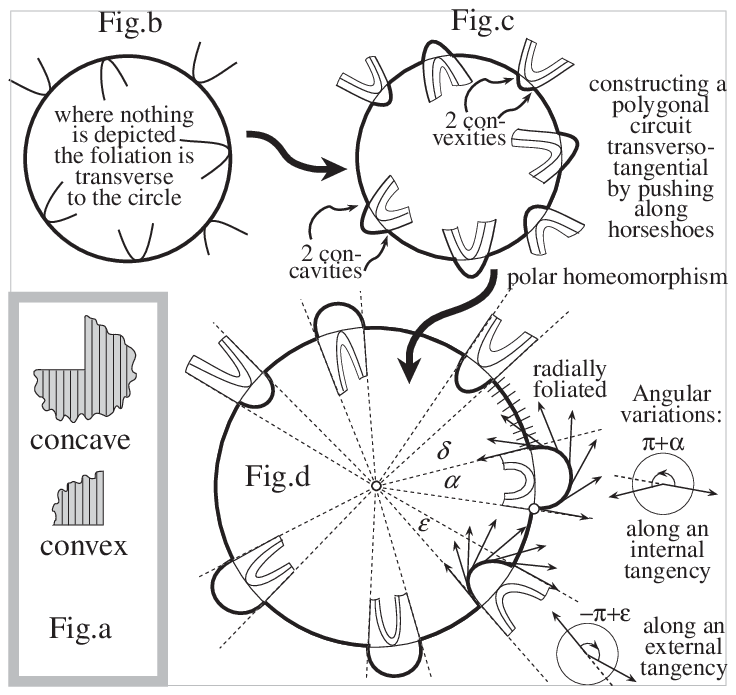,width=102mm}
  \caption{\label{Hamburg:fig}
  Naive approach to Hamburger's index formula}
\end{figure}

\begin{proof}[Proof of~\ref{Hamburger:lemma}] By general
position, we may assume that the circle $K$ has only finitely
many tangencies with the foliation of the Morse-type or
naively speaking {\it U-shaped\/} (cf.
Fig.\,\ref{Hamburg:fig}b). The ``U'' may go either inside or
outside the circle $K$ (locally at least), providing a
division in {\it internal\/} vs. {\it external\/} tangencies.
Choose about each
U-shaped leaf-arc tangent to $K$ a little tube (foliated box)
resembling a `horseshoe' (Fig.\,\ref{Hamburg:fig}c). By
pushing the circle $K$ inside resp. outside for each external
resp. internal tangencies gives the existence of a circuit
\hbox{$K'=:P$} of the desired type (again
Fig.\,\ref{Hamburg:fig}c, thick line).

We aim to compute the {\it Poincar\'e index\/}, which by
definition is the total angular variation of the tangent
during a complete circulation around any circle enclosing the
singularity (up to division by $2\pi$).

To this end we shall for simplicity assume (yet without
justification) that modulo a homeomorphism the configuration
can be normalized to one of the geometric type where the
horseshoes are so to speak in polar coordinates
(Fig.\,\ref{Hamburg:fig}d). On this picture it is further
assumed that the foliation is radial throughout the circular
segments.

As on
Fig.\,\ref{Hamburg:fig}d, let $\alpha_i$, resp.
$\varepsilon_j$ be the angles swept out by internal tangencies
resp. external ones, and $\delta_k$ be the remaining angles
corresponding to the arcs of
$K\cap K'$ which are radially foliated.
We have trivially $\sum_{i} \alpha_i+\sum_{j}
\varepsilon_j+\sum_{k} \delta_k=2\pi$.

An internal tangency, whose horseshoe sweeps out an angle
$\alpha_i$, causes an angular variation of $\pi+\alpha_i$ for
the turning tangent (compare Fig.\,\ref{Hamburg:fig}d).
Likewise an external tangency of angle $\varepsilon_j$ implies
a variation of $-\pi+\varepsilon_j$. Finally a circular
portion of the circuit
offering an angle of $\delta_k$
contributes to
a variation
of $\delta_k$.

Thus the total angular variation of the turning tangent(s) is
\begin{align*}
\sum_i (\pi+\alpha_i)+\sum_j (-\pi+\varepsilon_j)+\sum_k
\delta_k&= {I} \pi - E \pi+\underbrace{\textstyle\sum_i
\alpha_i+\sum_j \varepsilon_j+\sum_k \delta_k}_{=2\pi},
\end{align*}
where $I$, $E$ are the number of internal (resp. external)
tangencies. Since
the Poincar\'e index $j$ is the angular variation (divided by
$2\pi$), it follows that
\begin{equation}\label{Eq:Poinc-Bendi-Hamburg-index-formula}
j=1+\frac{I-E}{2}.
\end{equation}
Finally, from Fig.\,\ref{Hamburg:fig}c we have $2I=c'$, as
each internal tangency produces 2 concavities on the circuit
$K'$ deformed by the horseshoes, and likewise $2E=c$ as each
external tangency contributes to 2 convexities. This proves
formula \eqref{Hamburger:formula}.
\end{proof}

With Hamburger's formula \eqref{Hamburger:formula} we are
ready to complete the proof of
(\ref{index-small-if-no-loops}):

\medskip
\begin{proof}[Proof of Lemma~\ref{index-small-if-no-loops}]
Choose a polygonal circuit $K'$ as in (\ref{Hamburger:lemma}),
cf. also Fig.\,\ref{Hamburg:fig}c (thick line). We shall
describe two surgical processes on the circuit $K'$
diminishing the number $c'$ of concavities so as to make it
eventually equal to $0$, thereby proving our claim ($j\le 1$)
in the light
of Hamburger's formula $j=1-\frac{c-c'}{4}$.

As usual the proof involves some pictures. The goal is to kill
by surgery the
2 concavities generated by an inner tangency.

\begin{figure}[h]
\centering
    \epsfig{figure=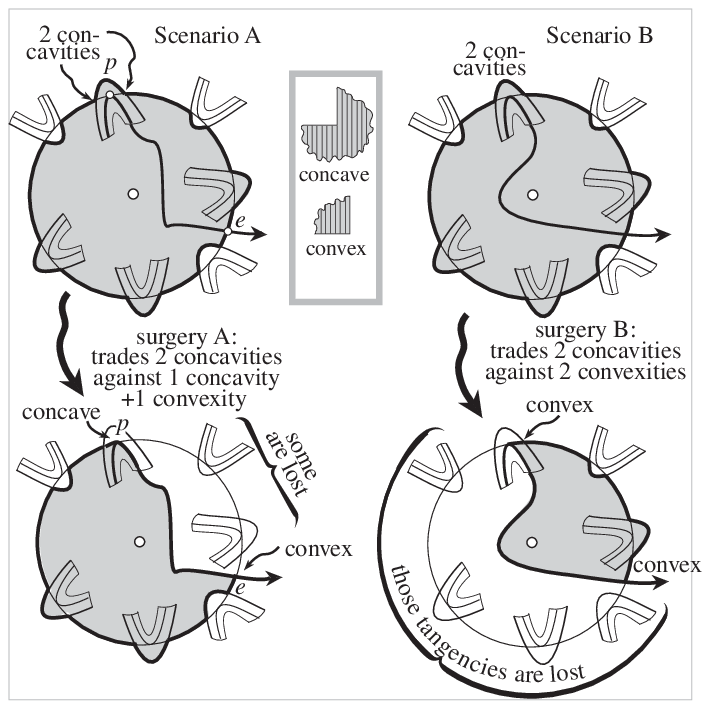,width=102mm}
  \caption{\label{Surgery:fig}
  Surgeries decreasing the number of concavities}
\end{figure}

If there is no inner tangencies,
we are done. Else, fix an inner tangency point $p$. If we
extend the U-shaped arc of leaf
emanating from $p$, then by assumption (no loops) one at least
of both ends must go to infinity. Otherwise a
Poincar\'e-Bendixson argument implies that the leaf starts
spiraling towards an asymptotic circle ({\it cycle limite})
which by Schoenflies must enclose the puncture. In this case
the index computed w.r.t. this circle is clearly $1$ (either
via Hamburger's formula or by the very definition of the
index). Hence the appropriate semi-leaf emanating from $p$
must eventually leave the bounded domain interior to $K'$
at some escape-point, say $e$. Close the segment of semi-leaf
$pe$ by the (unique) arc $A$ of $K'$ joining $p$ to $e$ such
that $pe+A$ is not null-homotopic (in the punctured plane).
This defines a new Jordan curve $K''=:J$.

As shown on Fig.\,\ref{Surgery:fig} two scenarios are possible
depending on whether the leaf-arc $pe$ closed by the sub-arc
of $K'$ circulated along the
clockwise orientation of the circle $K$ or $K'$ encloses or
not the puncture.

In the first case (Scenario A on Fig.\,\ref{Surgery:fig}), we
observe that the new Jordan curve $J$ is an admissible
polygonal circuit where $p$ is concave and $e$ is convex.
Hence the 2 concavities near $p$ on $K'$ are traded against 1
convexity (at $e$) and 1 concavity at $p$. Of course during
the process $K'\to J$ we may
loose several tangencies. In any event, the new numbers
$(c_1,c'_1)$ of convex resp. concave vertices (w.r.t. $J$)
satisfy $c_1\le c+1$ and more importantly $c'_1\le c'-1$.

In the second case (Scenario B on Fig.\,\ref{Surgery:fig}), we
observe that the 2 concavities are traded against 2
convexities, yielding thereby $c_1\le c+2$ and $c'_1\le c'-2$.

In both scenarios the number of concavities $c'$ decreases
under the surgery $K'\to J$, and after finitely many
iterations reaches (ineluctably) the value $0$. This completes
the proof of $j\le 1$ in view of Hamburger's formula.
\end{proof}

\subsection{Foliated index formula (Poincar\'e,
von Dyck, Brouwer, Hada\-mard, Hamburger, Kneser, Hopf,
Lefschetz)}

For a smooth vector field with isolated singularities on a
closed manifold, there is a well-known (remarkable) identity
between the total sum of the indices
at the  singularities
and the Euler characteristic of the manifold. This is known as
the {\it Poincar\'e-Hopf index formula}.
Its
intricate history may additionally involve Gauss, Cauchy,
Kronecker 1869, Poincar\'e 1881--1885 (\cite{Poincare_1881},
\cite{Poincare_1885}), Dyck 1888 \cite[p.\,501]{Dyck_1888},
Brouwer and Hadamard circa 1910 \cite{Hadamard_1910}, etc., up
to Hopf 1926 \cite{Hopf_1926}, not forgetting Lefschetz for
closely related works and the exposition in Alexandroff--Hopf
1935 \cite{Alexandroff-Hopf_1935}. (Again we may refer to
Mawhin 2000 \cite{Mawhin_2000} for a thorough historical
discussion, cf. also Hopf 1966 \cite{Hopf_1966}.)

For our application (\ref{Euler-obstruction:prop}) we
need only the surface case. Besides, there is a well-known
formulation of the index formula for line-fields or foliations
(compare e.g., Hopf 1946--56 \cite[p.\,113, 2.2 Thm
II]{Hopf_1946_1956} or Spivak 1975--79
\cite[p.\,331]{Spivak_1975-1979}):

\begin{theorem}\label{index-formula} For a foliated closed
surface $F$ with isolated singularities, the total sum of the
indices is equal to the Euler characteristic of the surface:
\begin{equation}
\textstyle\sum_{p\in F} j(p)=\chi(F).
\end{equation}
\end{theorem}

We present 2 proofs of this superb
theorem (in Lefschetz's 1967 appreciation
\cite[p.\,120]{Mawhin_2000}). The first is mostly inspired
form Poincar\'e's original paper of 1885. The second (easier
to find alone) gives only a reduction to the vector field
case, thus not very insightful, so safely to be skipped.

\smallskip
{\small

{\sc Historical quiz.} It is not perfectly clear to the
writer, who first formulated the index formula for foliations
(\ref{index-formula}). Of course, loosely speaking it is
Poincar\'e 1885 \cite[p.\,203--8]{Poincare_1885}, since the
foliated case is very akin to the flow case, either by passing
to a double cover orienting the foliation\footnote{This
standard technique is implicit (already) in Kneser 1921
\cite[p.\,85]{Kneser_1921}} or by noticing that Poincar\'e's
argument transposes better than {\it mutatis mutandis} to the
foliated situation, as practically nothing must be changed to
it (compare the next section). Strictly speaking, the first
source
might be Hamburger 1924 \cite[p.\,58--62]{Hamburger_1924},
where {\it a simple proof under very general assumptions is
given} (according to Hamburger 1940 \cite[Footnote~4,
p.\,64]{Hamburger_1940}). Much of this quiz is clarified by
reading Hamburger 1924 \cite[Fu{\ss}note 17,
p.\,57]{Hamburger_1924} (reproduced below as {\bf [Ham]}),
where
Kneser is mentioned as being also well aware of the foliated
index formula, cf. Kneser 1921 \cite[p.\,83]{Kneser_1921}
(={\bf [Kne]} below), where however [by a little
inadvertence?] the semi-integral nature of the index is not
emphasized, and
the reader is ``just'' referred to Dyck 1888 \cite{Dyck_1888}.
Here are the
relevant extracts:

\def\ll{\hskip1.5pt}

 {\bf [Ham]} $^{17)}$ Nach Abschlu{\ss} meines Manuskriptes erfuhr ich
durch eine freundliche briefliche Mitteilung von Herrn
{H{\ll}e\,l{\ll}l{\ll}m{\ll}u{\ll}t{\ll}h{\ll}
K{\ll}n{\ll}e{\ll}s{\ll}e{\ll}r}, da{\ss} dieser auch im
Besitze eines Beweises f\"ur die
P{\ll}o{\ll}i{\ll}n{\ll}c{\ll}a{\ll}r{\ll}\'e{\ll}sche Formel
in sehr allgemeinen F\"allen ist. Seine Methode, den Index
$i_{\varkappa}$ der Singularit\"at zu bestimmen, ist der
Methode der vorliegenden Note sehr \"ahnlich. Vergleiche die
Andeutungen der K{\ll}n{\ll}e{\ll}s{\ll}e{\ll}r{\ll}schen
Methoden und Ergebnisse in dem Verhandlungsberichte
des Naturforschertages in Jena. Jahresb. d. D.\,M.\,V. 30
(1921),
S.\,84.

 {\bf [Kne]} Bei Kurvenscharen mit Singularit\"aten wird jedem
Ausnahmepunkt eine Zahl zugeordnet, und $k$ ist gleich der
Summe dieser Zahlen, vgl. Dyck, Math. Ann.~32. }


\subsection{Poincar\'e's argument of 1885}

The proof presented below is, apart from minor variations (to
suit modern conventions), a `condensed' copy of Poincar\'e's
original argument (compare \cite[1885,
p.\,203--8]{Poincare_1885}), which is already announced in a
{\it Comptes Rendus\/} Note of 1881 \cite{Poincare_1881}. In
fact the latter assumes orientability of the surface and of
the foliation (i.e., argues with vector fields), yet
his argument
works without those provisos.
(Overlooking
the details of Poincar\'e's argument, one gets the wrong
impression that he only establishes the formula for special
singularities ({\it cols, n{\oe}uds} and {\it foyers}), yet on
p.\,208 (of {\it loc.\,cit.})  the general
case is established.)

\medskip
\begin{proof}[Proof of \ref{index-formula}] We aim to
compute the sum of the indices at the singularities,
relating this to the Euler characteristic
$\chi=\sigma_0-\sigma_1+\sigma_2$, where $\sigma_i$ is the
number of $i$-simplices ($i=0,1,2$) of any triangulation. (As
usual such simplices are resp. termed {\it vertices}, {\it
edges} and {\it triangles}.) By general position, we may
triangulate the surface $F$ so that each singularity lies in
the interior of a triangle, and each circuit bounding a
triangle has only finitely many tangencies with the foliation.
By invariance under deformation (homotopy\footnote{Jargon
coined by Dehn--Heegaard in 1907 ({\it sauf erreur!\/}).}), we
can extend the summation to all triangles (precisely their
boundaries), as those
lacking singularities
will not contribute (to the total sum of indices). (Of course,
to triangulate concrete pretzels (spheres with handles) as
well as their non-orientable avatars (cross-capped spheres),
one does not need Rad\'o's triangulation theorem
\cite{Rado_1925}, but it is still somehow implicitly used to
classify surfaces.)

Recall, from
Eq.\,\eqref{Eq:Poinc-Bendi-Hamburg-index-formula}, that the
index is given by Hamburger's formula (1922) (already
in Bendixson 1901 \cite[p.\,39]{Bendixson_1901}, and
even
in Poincar\'e 1885
\cite[p.\,203]{Poincare_1885}):
\begin{equation}\label{Eq:Poincare_et-ali}
j=1+\frac{I-E}{2},
\end{equation}
where $I$, $E$ are the number of internal resp. external
tangencies (or {\it contacts\/} as Poincar\'e calls them).

Poincar\'e observes \cite[p.\,207]{Poincare_1885} that each
tangency occurring along an edge of the triangulation is
counted {\it twice\/}, once as an internal and once as an
external contact, hence does not contribute to the total sum
of the indices.

Hence he needs only to worry about contacts occurring at
vertices (of the triangulation). Fix a vertex and assume it to
be the apex of $\nu$ triangles.
A leaf ({\it trajectoire} in his context) through the vertex
will traverse 2 triangles, while having an external contact
with the $\nu-2$ remaining  triangles. Thus,
\begin{equation}\label{Poinc:excess}
\sum_{triangles} (I-E)=\sum_{vertices}
(2-\nu)=2\sigma_0-3\sigma_2,
\end{equation}
since $\sum \nu= 3 \sigma_2$ as each $2$-simplex is counted
thrice (once for each of its apex).

Next, recall the following relation holding in any
triangulated closed surface
\begin{equation}\label{Eq:Descartes-Euler}
3\sigma_2=2\sigma_1.
\end{equation}
(This is easily verified by a double counting argument of the
incidence relation $I$ among pairs of
triangles given by {\it adjacency\/}: mapping an element of
$I$ to the common edge yields a 2-to-1 surjection to the
$1$-simplices, whereas projecting on the (first) factor yields
a 3-to-1 map to the set of all $2$-simplices.)

Thus,  we find for the total sum of indices (computed along
the contours of all triangles
of the triangulation):
\def\ziehen{\hskip-2pt}
\begin{align*}
\sum_{triangles} \ziehen
j&\buildrel{\eqref{Eq:Poincare_et-ali}}\over{=}\ziehen\sum_{triangles}
\bigl( 1+\frac{I-E}{2}\bigr)
\buildrel{\eqref{Poinc:excess}}\over{=}\sigma_2+\bigl(
\frac{2\sigma_0-3\sigma_2}{2}\bigr)
\buildrel{\eqref{Eq:Descartes-Euler}}
\over{=}\sigma_0-\sigma_1+\sigma_2=\chi(F).
\end{align*}
\vskip-15pt
\end{proof}

\subsection{Reduction of the index formula to the flow case}

\begin{proof}[Another proof of \ref{index-formula}] Now, we
just recall a formal reduction to the classical index formula
for flows (taking the latter's validity for granted via some
external source of your preference).

If the foliation $\cal F$ is orientable, then there is a
compatible flow by \Kerek-Whitney \cite{Kerekjarto_1925},
\cite{Whitney33}, and we assume this case settled.

If not, $\cal F$ determines a double cover $F^*\to F$ over
which the lifted foliation $\cal F^*$ is orientable. More
slowly, one first removes the singular set $S\subset F$ of the
foliation, and local orientations of leaves defines a double
cover $\Sigma\to F-S$, which
compatifies as a branched cover $\pi\colon \Sigma^*\to F$ by
filling over the punctures via Riemann's trick (cf. e.g.
\cite[2.16]{Gabard_2011_Ebullition}). Then by the
Riemann-Hurwitz formula:
\begin{equation}\label{Riemann-Hurwitz}
\chi(\Sigma^*)=2 \chi (F)- \deg(R),
\end{equation}
where $\deg(R)$ counts the ramification.

Each singularity of the foliation $\cal F$ is either {\it
orientable\/} ($\cal O$) or not ($\cal N$), yielding a
partition $S={\cal O} \sqcup \,{\cal N}$.
Singularities in $\cal N$ (non-orientable) are precisely those
responsible of the ramification so that $\# ({\cal
N})=\deg(R)$.  For $p\in \cal O $, let
$\pi^{-1}(p)=\{q_1,q_2\}$ and for $p\in \cal N $, let
$\pi^{-1}(p)=\{p^{*}\}$.

For a non-orientable singularity, recall the following
relation
$$
j=\frac{i+1}{2},
$$
between its index $j$ and the index $i$ of the lifted
orientable foliation  (compare Lemma~\ref{Hamburg:lifted}
below)

Finally, computing the total sum of the indices we find
\begin{align*}
\sum_{p\in S} j(p, \cal F)&=\sum_{p\in {\cal O}} j(p, {\cal
F})+\sum_{p\in {\cal N}} j(p, {\cal F})\cr
&=\frac{1}{2}\Bigl(\sum_{q\in \pi^{-1}{\cal O}} i(q, {\cal
F}^*)\Bigr)+\sum_{p^*\in \pi^{-1}{\cal N}} \frac{i(p^*, {\cal
F}^*)+1}{2}\qquad \text{(Lemma~\ref{Hamburg:lifted})}\cr
&=\frac{1}{2}\Bigl(\sum_{q\in \pi^{-1}{S}} i(q, {\cal
F}^*)+\deg(R)\Bigr)\cr
&=\frac{1}{2}\bigl(\chi(\Sigma^*)+\deg(R)\bigr) \qquad
\text{(Poincar\'e's index formula)}\cr &=\chi(F) \hskip3cm
\text{(by Riemann-Hurwitz \eqref{Riemann-Hurwitz})}.
\end{align*}
\vskip-20pt
\end{proof}

\begin{lemma}\label{Hamburg:lifted}
Given a non-orientable isolated singularity $p$ of a foliation
$\cal F$ on an elementary (piece of) surface $F\approx {\Bbb
R}^2$,  let $\pi\colon F^{*}\to F$ be the double cover
orienting the foliation and let ${\cal F}^*$ be the lifted
orientable foliation. Then
the index $i(p^*, {\cal F}^*)$ of the lifted foliation at the
unique point $p^*$ lying above $p$ and the original index
$j(p, {\cal F})$ downstairs are related by:
\begin{equation}
j(p, {\cal F})=\frac{i(p^*, {\cal F}^*)+1}{2}
\end{equation}
\end{lemma}

\begin{proof}
As we already used it (twice!), we permit us to deduce this
from Hamburger's formula
(although a proof from the scratch definition of the index
might also be possible, cf. for this Spivak \cite[Lemma~19,
p.\,328]{Spivak_1975-1979}). Indeed choose $P$ a polygonal
circuit as in (\ref{Hamburger:lemma}) enclosing the point $p$,
and denote again by $c$ resp. $c'$ the number of convex vs.
concave vertices. By covering space theory, the map $\pi$ is
topologically equivalent to $z\mapsto z^2$ in complex
coordinates on the punctured complex plane ${\Bbb C}^*$.
Lifting the circuit $P$ to the covering $F^*$ yields an
admissible circuit $P^{*}=\pi^{-1}(P)$
having
doubled quantities of convexities resp. concavities, i.e.
$c_*=2c$ and $c_*'=2c'$. Plugging into Hamburger's formula
(\ref{Hamburger:lemma}) gives the asserted
relation:
\begin{align*}
i(p^*,{\cal F}^*)=1-\frac{c_*-c_*'}{4}&=1-\frac{2c-2c'}{4}\cr
&=2\bigl( 1-\frac{c-c'}{4}\bigr)-1
=2j(p, {\cal F})-1.
\end{align*}
\vskip-20pt
\end{proof}

\medskip

{\small {\bf Acknowledgements.} The author wishes to thank
Andr\'e Haefliger, Claude Weber, David Gauld,
Jean-Claude Hausmann, Daniel Asimov and Mathieu Baillif  for
stimulating discussions.}

{\small

}

{
\hspace{+5mm} 
{\footnotesize
\begin{minipage}[b]{0.6\linewidth} Alexandre
Gabard

Universit\'e de Gen\`eve

Section de Math\'ematiques

2-4 rue du Li\`evre, CP 64

CH-1211 Gen\`eve 4

Switzerland

alexandregabard@hotmail.com
\end{minipage}
\hspace{-25mm} }

\end{document}